\documentclass[12pt,a4paper]{amsart}
\usepackage{amsmath,amsthm,amssymb,latexsym,xy,amssymb,enumerate,epic}
\usepackage{enumerate}
\xyoption{all} 

\newtheorem{theorem}{Theorem}
\newtheorem{lemma}[theorem]{Lemma}
\newtheorem{corollary}[theorem]{Corollary}
\newtheorem{proposition}[theorem]{Proposition}

\newtheorem{remark}[theorem]{Remark}

\newcommand{\tto}{\twoheadrightarrow}

\begin{document}

\title{Double Catalan monoids}
\author{Volodymyr Mazorchuk and Benjamin Steinberg}
\date{\today}

\begin{abstract}
In this paper we define and study what we call the double Catalan monoid.
This monoid is the image of a natural map from the $0$-Hecke monoid to
the monoid of binary relations. We show that the double Catalan monoid provides
an algebraization of the (combinatorial) set of $4321$-avoiding permutations
and relate its combinatorics to various off-shoots of both the combinatorics
of Catalan numbers and the combinatorics of permutations. In particular,
we give an algebraic interpretation of the first derivative of the
Kreweras involution on Dyck paths, of $4321$-avoiding involutions 
and of recent results of Barnabei {\em et al.} 
on admissible pairs of Dyck paths. We compute a presentation and determine 
the minimal dimension of an effective representation for the
double Catalan monoid. We also determine the minimal 
dimension of an effective representation for the $0$-Hecke monoid. 
\end{abstract}

\maketitle

\section{Introduction and description of the results}\label{s1}

The $0$-Hecke monoid $\mathcal{H}_n$ is the monoid obtained by specializing the defining
relations of the Hecke algebra at $q=0$ 
(and changing signs of the canonical generators to get rid of negatives). 
This is a classical object of study in both representation theory and combinatorics with
many applications, see for example \cite{No,Ca,RS,McN,Fa,HNT,NT2} and references therein.
An important quotient of the $0$-Hecke monoid, as observed in \cite{GM2}, is the 
monoid of all order preserving and non-decreasing transformations on $\{1,2,\dots,n\}$,
also known the {\em Catalan monoid} to emphasize that its cardinality is given by the
Catalan numbers. The presentation of $\mathcal{H}_n$ is very symmetric, but a large 
portion of this symmetry is lost by going to the Catalan monoid, as the latter corresponds
to choosing a ``one-way'' orientation on the Dynkin diagram. In other words, there
are two different quotients of $\mathcal{H}_n$, both isomorphic to the Catalan monoid,
corresponding to two different choices of the orientation.

The main object of study in the present paper is what we call the {\em double 
Catalan monoid}. It is defined as the image of $\mathcal{H}_n$ in the
semigroup $B_n$ of binary relations under the natural map sending a generator of  $\mathcal{H}_n$
to the semiring sum (inside of $B_n$) of the identity and the simple transposition of the
symmetric group corresponding to this generator. Equivalently, the double Catalan monoid can 
be described as the diagonal image of $\mathcal{H}_n$ in the direct sum of two ``opposite'' 
Catalan quotients, corresponding to the two different choices of orientation mentioned above.  

The generators of the double Catalan monoid, as described above, appear in various
guises in many rather different contexts. To start with, one can observe that 
the sum of the identity and a simple reflection is an element of the  Kazhdan-Lusztig basis
in the group algebra of $S_n$, see \cite{KL}. The binary relation representing this sum
appears in the theory of factor-powers of symmetric groups, see \cite{GM00}. In the theory of 
random walks, one often works with a  ``lazy'' version of a walk to remove periodicity phenomena.  
This amounts to standing still with probability $\frac{1}{2}$ and following the original random walk 
with probability $\frac{1}{2}$.  The operator $\frac{1}{2}(\mathrm{id}+s_{i})$ 
can be viewed then as a lazy random walk operator. Its ``booleanization'' is exactly a generator
of the double Catalan monoid.

The paper is structured as follows.
In Section~\ref{s2} we recall the definition of the $0$-Hecke monoid and discuss a combinatorial 
(or semigroup-theoretic) and a geometric realization of this monoid. The double Catalan monoid is
introduced in Section~\ref{s3}. It is defined as a submonoid of the monoid of binary relations.
We also give an alternative realization of the double Catalan monoid as a quotient of the 
$0$-Hecke monoid. To some extent the latter realization ``restores'' the symmetry lost under the
projection of the $0$-Hecke monoid onto the Catalan monoid. Section~\ref{s4} interprets several
combinatorial results for $4321$-avoiding permutations in terms of the algebraic structure of the
double Catalan monoid. This includes a bijection between $4321$-avoiding permutations and elements
of the double Catalan monoid, a bijection between $4321$-avoiding involutions and self-dual
elements of the double Catalan monoid, and an algebraic interpretation of the first derivative
of the Kreweras involution on Dyck paths and of the description of admissible pairs of Dyck paths in the
sense of \cite{BBS}. In Section~\ref{s5} we give a presentation of the double Catalan monoid
by exploiting a result of \cite{T} characterizing vexillary permutations. In Section~\ref{s6}
we propose generalizations of both the Catalan and double Catalan monoids to other Coxeter groups
and parabolic subgroups (with our original definition corresponding to the case of the symmetric
group $S_n$ and its maximal parabolic subgroup $S_{n-1}$).

In Section~\ref{s5} we prove that the minimal dimension (over any field) of an injective representation 
of the $0$-Hecke monoid associated to a finite Coxeter group $W$ equals sum of the indices of its maximal
parabolics minus the rank of $W$. From this we deduce that the minimal dimension (over any field) of 
an injective representation of the double Catalan monoid corresponding to $S_n$ is $2n-2$. Note that in
general the question of computing the minimal dimension of an injective representation of a given
monoid is very hard and there are few known techniques for doing this. The key ingredient in
our approach is a reinterpretation of the combinatorics of Bruhat quotients in terms of the algebraic
structure of left ideals of $0$-Hecke monoids (a similar interpretation appears in \cite{St}).
\vspace{5mm}

\noindent
{\bf Acknowledgements.} The paper was written during a visit of the
second author to Uppsala University. The financial support and 
hospitality of Uppsala University are gratefully acknowledged.
The first author was partially supported by the
Royal Swedish Academy of Sciences and the Swedish Research Council.
The second author was partially supported by NSERC. 
We thank the referee for valuable comments.

\section{$0$-Hecke monoids}\label{s2}

\subsection{Classical definition}\label{s2.1}

Let $\mathbb{N}$ denote the set of all positive integers. 
For $n\in\mathbb{N}$ set $\mathbf{N}:=\{1,2,\dots,n\}$, 
$\mathbf{N}':=\{1,2,\dots,n-1\}$ and let
$S_n$ be the symmetric group on $\mathbf{N}$. We will use the {\em one-line notation}
for elements of $S_n$ and write $a_1a_2\dots a_n$ for the permutation
\begin{displaymath}
\left(\begin{array}{cccc}1&2&\dots&n\\a_1&a_2&\dots&a_n\end{array}\right). 
\end{displaymath}
As usual, for $i\in\mathbf{N}'$
we denote by $s_i$ the simple transposition $(i,i+1)$ of $S_n$. We denote by
$\mathrm{id}$ the identity element of $S_n$. Then the $s_i$ are Coxeter generators of
$S_n$ and satisfy the following set of defining relations (for all appropriate
$i,j\in\mathbf{N}'$):
\begin{displaymath}
s_i^2=\mathrm{id};\quad\quad\quad s_is_j=s_js_i,\quad i\neq j\pm 1;\quad\quad\quad
s_is_{i+1}s_i=s_{i+1}s_is_{i+1}.
\end{displaymath}
The corresponding $0$-Hecke monoid $\mathcal{H}_n$ is then defined as generated by
elements $e_i$, $i\in\mathbf{N}'$, subject to the following relations (for all 
appropriate $i,j\in\mathbf{N}'$):
\begin{equation}\label{eq1}
e_i^2=e_i;\quad\quad\quad e_ie_j=e_je_i,\quad i\neq j\pm 1;\quad\quad\quad
e_ie_{i+1}e_i=e_{i+1}e_ie_{i+1}.
\end{equation}
Reversal of words induces an involution on $\mathcal H_{n}$ that we term {\em canonical}.

For $w\in S_n$ with reduced decomposition $w=s_{i_1}s_{i_2}\cdots s_{i_k}$ consider
the element $z_{w}:=e_{i_1}e_{i_2}\cdots e_{i_k}\in \mathcal{H}_n$. Then $z_w$ does 
not depend on the reduced decomposition of $w$ and the map $w\mapsto z_w$ is bijective
(see \cite[Theorem~1.13]{Ma}). For $w\in S_n$, the corresponding 
element $z_w$ is an idempotent if and only if $w$ is the longest element
of some parabolic subgroup of $S_n$, see \cite[Lemma~2.2]{No}. In particular,
$\mathcal{H}_n$ has exactly $2^{n-1}$ idempotents.

\subsection{Subset realization}\label{s2.2}

Let $\mathcal{P}(S_n)$ denote the power semigroup of $S_n$, that is, the set of all subsets 
of $S_n$.  It is a monoid with respect to the operation $A\cdot B=\{ab:a\in A, b\in B\}$.
The following statement was observed by several people (in particular, S.~Margolis and the 
second author mention this without proof in \cite{MS}):

\begin{proposition}\label{prop1}
The submonoid $T$ of $\mathcal{P}(S_n)$ generated by  
$\{\mathrm{id},s_i\}$, for $i\in\mathbf{N}'$, is isomorphic to $\mathcal{H}_n$.
\end{proposition}

\begin{proof}
Set $A_i:=\{\mathrm{id},s_i\}$. It is straightforward to verify that the $A_i$
satisfy the following relations: 
\begin{displaymath}
A_i^2=A_i;\quad\quad A_iA_j=A_jA_i,\quad i\neq j\pm 1;\quad\quad
A_iA_{i+1}A_i=A_{i+1}A_iA_{i+1}.
\end{displaymath}
This means that the map $e_i\mapsto A_i$ extends uniquely to a surjective homomorphism
$\varphi\colon \mathcal{H}_n\tto T$. 

The subword property of the Bruhat order on $S_n$ implies that, for every $w\in S_n$,
the subset $\varphi(z_w)$ coincides with the principal ideal of $S_n$ (with respect 
to the Bruhat order) generated by $w$. This implies that $\varphi$ is injective
and hence an isomorphism.
\end{proof}

In other words, $\mathcal{H}_n$ can be viewed as the monoid of principal Bruhat
order ideals under the usual multiplication of subsets.  Recall that a monoid $M$ is said to 
be $\mathcal{J}$-trivial if $MaM=MbM$ implies $a=b$.

\begin{corollary}\label{cor2}
The monoid $\mathcal{H}_n$, realized as a submonoid of $\mathcal{P}(S_n)$ as above,
is an ordered monoid (by inclusion) in which $\{\mathrm{id}\}$ is the smallest
element. In particular, $\mathcal{H}_n$ is $\mathcal{J}$-trivial. 
\end{corollary}

\begin{proof}
The first claim is obvious. The ``in particular'' statement follows
from \cite{ST}. 
\end{proof}

The order on $\mathcal{H}_n$ described in Corollary~\ref{cor2} corresponds to the
Bruhat order under the identification given by Proposition~\ref{prop1}
(that is $z_u\leq z_w$ if and only if $u\leq w$ in the Bruhat order).
A direct argument for $\mathcal{J}$-triviality of $\mathcal{H}_n$ can be found in
\cite{GM2}.

\subsection{Realization via foldings of the Coxeter complex}\label{s2.3}

The symmetric group $S_n$ gives rise to an example of a Coxeter group.  Every Coxeter group $W$ acts 
on a special simplicial complex $\Sigma(W)$, called its {\em Coxeter complex}.  The Coxeter generators act by reflecting over the walls of the complex.  The face poset of $\Sigma(W)$ is the set of all cosets of (standard) parabolic subgroups of $W$ ordered by reverse inclusion. For details concerning the theory of Coxeter complexes described below, see~\cite{AB}.

In the case of the symmetric group $S_n$, there is a simple combinatorial description of this complex.  The symmetric group acts on the $n$-simplex as its symmetry group by permuting the vertices and extending uniquely to an affine map.  The Coxeter complex $\Sigma(S_n)$ of $S_n$ is the barycentric subdivision of the boundary of the $n$-simplex.  Formally, the vertex set of $\Sigma(S_n)$ consists of all non-empty proper subsets of $\mathbf N$ (and so there are $2^{n}-2$ vertices). One should think of a subset as corresponding to the barycenter of the corresponding face of the $n$-simplex.  A simplex is then a set $\{F_1,\ldots,F_k\}$ of subsets forming a flag $F_1\subsetneq F_2\subsetneq \cdots \subsetneq F_k$.   Maximal faces are called chambers in this context.  The {\em fundamental chamber} $C$ is the flag 
\begin{displaymath}
\{1\} \subsetneq \{1,2\} \subsetneq\cdots\subsetneq  \{1,\ldots,n-1\}.
\end{displaymath}

Notice that the chambers are in bijection with permutations by listing the sequence of elements adjoined at each step followed by the missing element.  For example $\{3\}\subsetneq \{34\}\subsetneq\{341\}$ corresponds to the permutation $3412$.
In general, one can identify the face $F_1\subsetneq F_2\subsetneq\cdots \subsetneq F_k$ with the ordered set partition $(F_1,F_2\setminus F_1,\ldots, F_k\setminus F_{k-1},\mathbf N\setminus F_k)$ and so we can view the faces as ordered set partitions with at least two blocks.  Going down in the order corresponds to joining together two consecutive blocks.  The faces of the fundamental chamber are in bijection with subsets $J$ of $\mathbf N'$. If $J$ consists of $i_1<i_2<\cdots<i_k$, then the corresponding ordered set partition is 
\begin{equation}\label{standardface}
 (\{1,\ldots,i_1\},\{i_1+1,\ldots,i_2\},\ldots,\{i_k+1,\ldots,n\}).
\end{equation}
This face and all the elements of its orbit under $S_n$ are said to have {\em type} $J$.  Notice that the stabilizer of \eqref{standardface} is the parabolic subgroup generated by $s_i$, $i\in\mathbf N'\setminus J$.   For instance, the panel (codimension one) face of $C$ of type $\mathbf N'\setminus \{i\}$ is $(1,\ldots,i-1,\{i,i+1\},i+2,\ldots,n)$.

A simplicial endomorphism of $\Sigma(S_n)$ is said to be {\em type-preserving} if it preserves the type of each face.    In particular, such an endomorphism must preserve the dimension of each simplex as the dimension of a face of type $J$ is $|J|$.  Type-preserving endomorphisms are determined by their actions on chambers.  To make this more precise, consider the unoriented labeled Cayley graph $\Gamma$ of $S_{n}$ with respect to the Coxeter generators.  So the vertex set of $\Gamma$ is $S_{n}$ and there is an edge between $v$ and $w$ labeled by $s_j$ if $v^{-1}w=s_{j}$.  If we identify the chambers with vertices of $\Gamma$ and the panels with labeled edges by associating a face of type $\mathbf N'\setminus \{j\}$ with the edge labeled $s_j$, then one has that the monoid of type-preserving endomorphisms of $\Sigma(S_n)$ is isomorphic to the monoid of label-preserving endomorphisms of $\Gamma$.

The {\em wall} $H_i$ of $\Sigma(S_n)$ associated to the simple reflection $s_i$ is the subcomplex fixed by $s_i$.  It consists of all ordered set partitions in which $\{i,i+1\}$ are in the same block.  A {\em folding} of $\Sigma(S_n)$ is an idempotent type-preserving endomorphism $\varphi$ such that $|\varphi^{-1}(\varphi(C'))|=2$ 
(here $\varphi^{-1}$ stands for the inverse image of a mapping) for each chamber $C'$.

To every wall $H_i$ is associated a unique folding $\varphi_i$ such that $\varphi_i(C)=s_iC$.  Intuitively, it fixes the half-space (or root) containing $s_iC$ and reflects the half-space containing $C$.  Its action on a chamber $wC$ is specified by the formula
\begin{displaymath}
 \varphi_i(wC)=\begin{cases}
            s_iwC, & \text{if}\ \mathfrak{l}(w)< \mathfrak{l}(s_iw);\\
wC & \text{if}\ \mathfrak{l}(w)> \mathfrak{l}(s_iw).
           \end{cases}
\end{displaymath}

It now follows immediately that $\mathcal H_n$ is isomorphic to the monoid generated by the foldings $\varphi_i$, with $i\in \mathbf N'$, with the action on chambers being isomorphic to the regular action of $\mathcal H_n$ on itself.  In fact, the labeled Cayley graph $\Gamma$ of $S_{n}$ can be obtained from the Cayley digraph of $\mathcal H_{n}$ by removing the loop edges and forgetting the orientation.  Thus the action of $\mathcal H_{n}$ on $\Gamma$ is essentially its natural action on its labeled Cayley digraph.

The action of the generators of $\mathcal H_n$ on the Coxeter complex $\Sigma(S_{n})$ is easy to describe.  If $F=(F_1,\ldots,F_k)$ is an ordered set partition, then $\varphi_i$ fixes $F$ unless $i$ and $i+1$ are in different blocks of $F$ and the block of $i$ comes before the block of $i+1$.  In this case, one transposes $i$ and $i+1$.  For example, $\varphi_1(\{1,3\},\{2,4\})=(\{2,3\},\{1,4\})$.

The observation that the $0$-Hecke monoid acts on the Coxeter complex can essentially be found in~\cite{HST} where it is phrased in the language of root systems and Tits cones (note that in this realization the action of the
$0$-Hecke monoid in not linear).

\section{Double Catalan monoids}\label{s3}

\subsection{Binary relations and Boolean matrices}\label{s3.1}

Denote by $B_n$ the semigroup of binary relations on $\mathbf{N}$.
This is an ordered monoid (with respect to inclusion).
The semigroup $B_n$ can be identified with the semigroup 
$M_n(\{0,1\})$ of all $n\times n$-matrices over the Boolean semiring $\{0,1\}$
in the following way: to $\xi\in B_n$ there correspond a matrix $(\xi_{i,j})$,
where $\xi_{i,j}=1$ if and only if $(i,j)\in\xi$.
This identification equips $B_n$ with the structure of a semiring. In what follows
we will freely use this identification and refer to elements of $B_n$ both as relations
and as boolean matrices, depending on which is more convenient. Denote by $\Phi\colon S_n\to B_n$ 
the usual embedding and by $(\cdot)^t\colon M_n(\{0,1\})\to M_n(\{0,1\})$ the transpose map.

For $\xi\in B_n$ and $j\in\mathbf{N}$ we set $\xi(j)=\{i:(i,j)\in\xi\}\subseteq \mathbf{N}$.
A binary relation $\xi$ is called {\em convex} provided that it 
is reflexive and  for every $i\in\mathbf{N}$
the sets $\xi(j)$ and $\xi^t(j)$ are intervals of $\mathbf{N}$
(we will call this condition the {\em interval condition}). 
Let $\mathcal{C}B_n$ denote the set of all
convex binary relations. The following statement is easy to check:

\begin{proposition}\label{prop3}
The set $\mathcal{C}B_n$ is a submonoid of $B_n$.
\end{proposition}

Clearly, the submonoid $\mathcal{C}B_n$  is stable under transpose.

\subsection{Definition of the double Catalan monoids}\label{s3.2}

The idea of the following definition comes from an attempt of ``booleanization''
of Pro\-po\-si\-ti\-on~\ref{prop1}. For $i\in\mathbf{N}'$ consider the element 
$\varepsilon_i:=\Phi(\mathrm{id})+\Phi(s_i)\in B_n$
(see the example in Figure~\ref{fig1}). Denote by $\mathcal{DC}_n$
the submonoid of $B_n$ generated by $\varepsilon_i$, $i\in\mathbf{N}'$.
We will call $\mathcal{DC}_n$ the {\em double Catalan monoid} (our motivation for
this name should become clear by the end of this section).
Note that each $\varepsilon_i$ is convex and hence
$\mathcal{DC}_n$ is a submonoid of $\mathcal{C}B_n$.
As $\varepsilon_i^t=\varepsilon_i$, the involution $(\cdot)^t$ restricts to
an involution on $\mathcal{DC}_n$.
From the definition it follows that $\mathcal{DC}_n$ is
a submonoid of the factor power of the symmetric group,
studied in \cite{GM00}.

\begin{figure} 
\begin{displaymath}
\varepsilon_1=\left(\begin{array}{cccc}1&1&0&0\\1&1&0&0
\\0&0&1&0\\0&0&0&1\end{array}\right)\quad
\varepsilon_2=\left(\begin{array}{cccc}1&0&0&0\\0&1&1&0
\\0&1&1&0\\0&0&0&1\end{array}\right)\quad
\varepsilon_3=\left(\begin{array}{cccc}1&0&0&0\\0&1&0&0
\\0&0&1&1\\0&0&1&1\end{array}\right)
\end{displaymath}
\caption{The elements $\varepsilon_1,\varepsilon_2,
\varepsilon_3\in B_4$}\label{fig1}
\end{figure}

\begin{remark}\label{rem5005}
{\rm 
It is often useful to remember the following combinatorial description of the
action of generators of $\mathcal{DC}_n$:
If $\xi\in B_n$, then the matrix $\varepsilon_i\xi$ (resp.\ $\xi\varepsilon_i$) is 
obtained from $\xi$ by replacing the $i$-th and the $i+1$-st rows (resp.\ columns)
of $\xi$ by their sum (over the Boolean semiring $\{0,1\}$).
}
\end{remark}

\subsection{Catalan quotients of $0$-Hecke monoids}\label{s3.3}

Denote by $\mathcal{C}_n^{+}$ and $\mathcal{C}_n^{-}$ the quotients
of $\mathcal{H}_n$ modulo the additional relations
$e_ie_{i+1}e_i=e_{i+1}e_{i}$ and $e_ie_{i+1}e_i=e_{i}e_{i+1}$
(for all appropriate $i$), respectively. We call these monoids
{\em Catalan quotients} of $\mathcal{H}_n$, as
$|\mathcal{C}_n^{+}|=|\mathcal{C}_n^{-}|=C_{n}=\frac{1}{n+1}\binom{2n}{n}$ 
is the $n$-th
Catalan number (see \cite{So} and \cite{GM2}). There is a more general
family of so-called {\em Kiselman quotients} of $\mathcal{H}_n$,
studied in \cite{GM2}.

\subsection{The enveloping Catalan monoid}\label{s3.4}

Recall that a transformation $\xi\colon\mathbf{N}\to\mathbf{N}$ is called {\em non-decreasing }
(resp. {\em non-increasing}) provided that for all $i\in\mathbf{N}$ we have
$i\leq \xi(i)$  (resp. $i\geq \xi(i)$).
Denote by $\mathbf{C}_n^+$ (resp.\ $\mathbf{C}_n^-$) the classical 
{\em Catalan monoid} of all order preserving and non-decreasing 
(resp.\ non-increasing) transformations on $\mathbf{N}$. The monoid $\mathbf{C}_n^+$
is ordered with respect to the pointwise ordering on functions.
Dually, the monoid $\mathbf{C}_n^-$ is ordered with respect to the opposite of
the pointwise ordering on functions.
%We consider $\mathbf{C}_n$ as a submonoid of $B_n$ in the obvious way.
We have $\mathbf{C}_n^+\cong \mathbf{C}_n^-\cong 
\mathcal{C}_n^{+}\cong \mathcal{C}_n^{-}$,
see \cite{So} and \cite{GM2}.
Define the {\em enveloping Catalan monoid} $\overline{\mathbf{C}}_n$
as $\mathbf{C}_n^{+}\times \mathbf{C}_n^{-}$. This is an ordered monoid
with the product order.

Let $\mathcal{PT}_n$ be the submonoid of $B_n$ consisting of all
partial transformations (i.e., $\xi\in B_n$ such that $|\xi(i)|\leq 1$
for all $i\in\mathbf{N}$, see \cite{GM}).
Define maps $\mathrm{max}\colon B_n\to \mathcal{PT}_n$ 
and $\mathrm{min}\colon B_n\to \mathcal{PT}_n$ as follows:
the relation $\mathrm{max}(\xi)$ 
(resp.\ $\mathrm{min}(\xi)$) contains $(i,j)$ if and only if
$i=\max(\xi(j))$ (resp.\ $i=\min(\xi(j))$). An example of how this works is
given in Figure~\ref{fig2}.  Our first essential observation is the following:

\begin{figure} 
{\small
\begin{displaymath}
\xi=\left(\begin{array}{cccc}1&1&1&0\\1&1&1&0
\\0&1&1&0\\0&0&1&1\end{array}\right)\quad
\mathrm{max}(\xi)=\left(\begin{array}{cccc}0&0&0&0\\1&0&0&0
\\0&1&0&0\\0&0&1&1\end{array}\right)\quad
\mathrm{min}(\xi)=\left(\begin{array}{cccc}1&1&1&0\\0&0&0&0
\\0&0&0&0\\0&0&0&1\end{array}\right)
\end{displaymath}
}
\caption{The transformations $\mathrm{max}$ and
$\mathrm{min}$}\label{fig2} 
\end{figure}

\begin{theorem}\label{thm4}
The map $\Theta\colon \mathcal{C}B_n\to \overline{\mathbf{C}}_n$,
$\xi\mapsto(\mathrm{max}(\xi),\mathrm{min}(\xi))$,
is an isomorphism of ordered monoids.
\end{theorem}

\begin{proof}
Every $\xi\in\mathcal{C}B_n$ is reflexive, which implies that, on the 
one hand,  both $\mathrm{max}(\xi)$ and $\mathrm{min}(\xi)$ are total
transformations of $\mathbf{N}$ and, on the other hand, that
$\mathrm{max}(\xi)$ is non-decreasing and $\mathrm{min}(\xi)$ is
non-increasing. Now the facts that $\mathrm{max}(\xi)\in\mathbf{C}_n^+$ and 
$\mathrm{min}(\xi)\in\mathbf{C}_n^-$ follow easily from the interval 
condition.  Therefore $\Theta$  is well-defined.

To show that $\Theta$ is bijective, we construct the inverse as follows:
Given $\alpha\in \mathbf{C}_n^+$ and $\beta\in \mathbf{C}_n^-$
we have $\alpha(i)\geq i\geq \beta(i)$ for all $i\in\mathbf{N}$.
define $\psi((\alpha,\beta))\in B_n$ as the unique binary relation such 
that  $(i,j)\in \xi$ if and only if $\alpha(j)\geq i\geq \beta(j)$.
It is easy to check that $\xi$ is, in fact, an element of
$\mathcal{C}B_n$ and that $\psi$ is the inverse of $\Theta$.

It is routine to verify that both $\Theta$ and $\psi$
are order preserving, so
it is left to check that $\Theta$ is a homomorphism. This amounts 
to checking that the maps $\mathrm{max}$ and $\mathrm{min}$ are 
homomorphisms when restricted to $\mathcal{C}B_n$. We will check it for
$\mathrm{max}$; for $\mathrm{min}$ one can use similar arguments.
Since for $\xi\in \mathcal{C}B_n$ the element $\mathrm{max}(\xi)$
is a total transformation, it is enough to check that every
$(i,j)\in \mathrm{max}(\xi)\mathrm{max}(\eta)$, 
$\xi,\eta\in \mathcal{C}B_n$, belongs 
to $\mathrm{max}(\xi\eta)$. Let $s\in\mathbf{N}$ be such 
that $(i,s)\in\mathrm{max}(\xi)$ and $(s,j)\in\mathrm{max}(\eta)$.
then $(i,s)\in\xi$ and $(s,j)\in\eta$ and thus $(i,j)\in\xi\eta$.
Assume $(i',j)\in \xi\eta$. Then there is
$t\in\mathbf{N}$ such that $(i',t)\in\xi$ and $(t,j)\in\eta$.
Then $t\leq s$ since $s=\mathrm{max}(\eta)(j)$. As
$\mathrm{max}(\xi)$ is order preserving, we get
$i'\leq i$. The claim follows.
\end{proof}

\subsection{Double Catalan monoids via $0$-Hecke monoids}\label{s3.5}

The following statement relates $\mathcal{H}_n$ to $\mathcal{DC}_n$:

\begin{proposition}\label{prop5}
There is a unique surjective homomorphism $\Psi\colon \mathcal{H}_n\to\mathcal{DC}_n$
of ordered monoids 
such that $\Psi(e_i)=\varepsilon_i$ for all $i\in\mathbf{N}'$.
\end{proposition}

\begin{proof}
The map $\Phi\colon S_n\to B_n$ extends to a semiring homomorphism
$\overline{\Phi}\colon \mathcal{P}(S_n)\to B_n$ (here $\mathcal{P}(S_n)$ is 
equipped with the obvious structure of a semiring with the union as addition).
Then $\overline{\Phi}(\{\mathrm{id},s_i\})=\varepsilon_i$ 
and the claim follows from Proposition~\ref{prop1}.
\end{proof}

From Proposition~\ref{prop5} and Subsection~\ref{s3.4} it follows that 
the monoid $\mathcal{DC}_n$ can be understood as the ``diagonal''
image of $\mathcal{H}_n$ in the enveloping Catalan monoid
$\overline{\mathbf{C}}_n$. In particular, the monoid
$\mathcal{DC}_n$ is the quotient of $\mathcal{H}_n$ modulo the
intersection of the kernels of canonical maps from 
$\mathcal{H}_n$ to $\mathcal{C}_n^+$ and $\mathcal{C}_n^-$.  Note that the mapping in Proposition~\ref{prop5} respects the canonical involutions.

For $w\in S_n$ define the {\em left-to-right maximum} transformation
$\alpha_w$ and the {\em right-to-left minimum} transformation $\beta_w$
of $\mathbf{N}$ for $i\in \mathbf{N}$ as follows (cf.~\cite{BBS}):
\begin{displaymath}
\begin{array}{rcl}
\alpha_w(i)&=&\max\{w(j):j=1,2,\dots,i\},\\ 
\beta_w(i)&=&\min\{w(j):j=i,i+1,\dots,n\}. 
\end{array}
\end{displaymath}
Obviously, we have $\alpha_w(i)\geq i\geq  \beta_w(i)$ for all $i\in \mathbf{N}$ and $\alpha_w\in \mathbf{C}_n^+$, $\beta_w\in \mathbf{C}_n^-$.
The next statement describes $\Psi(z_w)$ for all $w\in S_n$.

\begin{proposition}\label{prop6}
For every $w\in S_n$ the binary relation $\Psi(z_w)$ is the unique 
element in $\mathcal{C}B_n$ satisfying $\max(\Psi(z_w))=\alpha_w$ 
and $\min(\Psi(z_w))=\beta_w$.
\end{proposition}

\begin{proof}
We prove this by induction on the length $\mathfrak{l}(w)$ of $w$. 
If $\mathfrak{l}(w)=0$, then $w$ is the identity transformation of
$\mathbf{N}$ and $\Psi(z_w)$ is the identity binary relation. In this
case the claim is easy to check.

Assume now that the claim is true for some $w$ and $i\in\mathbf{N}'$
is such that $\mathfrak{l}(ws_i)>\mathfrak{l}(w)$. The latter is 
equivalent to $w(i)<w(i+1)$. Define $u:=ws_i$. Then $\alpha_w$ and
$\alpha_u$ agree for all $j\neq i$ and $\beta_w$ and $\beta_u$ agree 
for all $j\neq i+1$. We claim that
\begin{eqnarray}
\alpha_u(i)&=&\alpha_u(i+1)\,=\,\max\{\alpha_w(i),\alpha_w(i+1)\}\,=\,\alpha_w(i+1) ;\label{eq2}\\
\beta_u(i)&=&\beta_u(i+1)\,=\,\min\{\beta_w(i),\beta_w(i+1)\}\,\,\,=\,\beta_w(i).\label{eq3}
\end{eqnarray}
Indeed, if $\alpha_w(i)>w(i+1)$, then
$\alpha_w(i)=\alpha_w(i+1)=\alpha_u(i)=\alpha_u(i+1)$. 
If $\alpha_w(i)<w(i+1)$, then
\begin{displaymath}
\alpha_u(i)=\alpha_u(i+1)=w(i+1)=\alpha_w(i+1)>\alpha_w(i). 
\end{displaymath}
This implies formula \eqref{eq2} in both cases and the formula \eqref{eq3} 
is proved similarly. Now the induction step follows from the inductive 
assumption and Remark~\ref{rem5005}.
\end{proof}

In particular, it follows that the natural map from $\mathcal{H}_n$ to $\mathbf{C}_n^+$
(resp. $\mathbf{C}_n^-$) takes $z_w$ to $\alpha_w$ (resp. $\beta_w$).

\section{Combinatorics of double Catalan monoids}\label{s4}

\subsection{Projection onto the Catalan quotient}\label{s4.2}

Let us consider the natural projection $\mathfrak{p}\colon \overline{\mathbf{C}}_n\to \mathbf{C}_n^+$ .
Define $\overline{\mathfrak{p}}:=\mathfrak{p}\circ\Theta\circ\Psi\colon 
\mathcal{H}_n\to \mathbf{C}_n^+$. For $\alpha\in \mathbf{C}_n^+$ set
$\overline{\mathfrak{p}}_{\alpha}:=\{w\in S_n:\overline{\mathfrak{p}}(z_w)=\alpha\}$.

\begin{proposition}\label{prop11}
\begin{enumerate}[$($a$)$]
\item\label{prop11.1} The map $\overline{\mathfrak{p}}$ is surjective.
\item\label{prop11.2} For every $\alpha\in \mathbf{C}_n^+$ the set
$\overline{\mathfrak{p}}_{\alpha}$ contains 
a unique $321$-avoiding permutation $\pi_{\alpha}$.
\item\label{prop11.3} The element $\pi_{\alpha}$ is the unique minimal element
of $\overline{\mathfrak{p}}_{\alpha}$ with respect to the Bruhat order.
\item\label{prop11.4} For every $\alpha\in \mathbf{C}_n^+$ the set
$\overline{\mathfrak{p}}_{\alpha}$ contains  
a unique $312$-avoiding permutation $\pi'_{\alpha}$.
\item\label{prop11.5} The element $\pi'_{\alpha}$ is the unique maximal element
of $\overline{\mathfrak{p}}_{\alpha}$ with respect to the Bruhat order.
\item\label{prop11.6} The set $\overline{\mathfrak{p}}_{\alpha}$ is
the Bruhat interval between $\pi_{\alpha}$ and $\pi'_{\alpha}$.
\end{enumerate}
\end{proposition}

\begin{proof}
Given $\alpha\in \mathbf{C}_n^+$ define the value of $\pi_{\alpha}$ on $i\in\mathbf{N}$ 
recursively as follows:
\begin{displaymath}
\pi_{\alpha}(i)=
\begin{cases}
\alpha(i), & \text{if}\,\, \alpha(i)>\alpha(i-1);\\
\mathrm{min}(\mathbf{N}\setminus\{\pi_{\alpha}(j):j<i\}),&\text{else}. 
\end{cases}
\end{displaymath}
It is easy to check that $\pi_{\alpha}\in S_n$ and that $\overline{\mathfrak{p}}(\pi_{\alpha})=\alpha$,
which proves claim \eqref{prop11.1}. From the construction it also follows directly that
$\pi_{\alpha}$ is $321$-avoiding, which gives the existence part of claim \eqref{prop11.2}.
The uniqueness part of claim \eqref{prop11.2} is proved as in \cite[4.2]{Bo}.
Claim \eqref{prop11.4} follows from the bijection described in \cite[Lemma~4.3]{Bo}.

Assume that $w\in\overline{\mathfrak{p}}_{\alpha}$. If the element $w$ is not $321$-avoiding
(resp.\ $312$-avoiding), we can choose the corresponding $321$-pattern (resp.\ $312$-pattern)
$w(i),w(j),w(k)$ for some $i<j<k$ such that
\begin{displaymath}
\text{ either } w(s)<\min(w(j),w(k)) \text{ or } w(s)>\max(w(j),w(k))
\end{displaymath}
for all $s$ such that $j<s<k$. Then, swapping
$w(j)$ and $w(k)$ changes the $321$-pattern into a $312$-pattern, and vice versa.
At the same time, going from the $321$-pattern to a $312$-pattern we produce a smaller element
with respect to the Bruhat order, and vice versa. Moreover, this transformation clearly does not
affect $\alpha_w$. This implies claims  \eqref{prop11.3} and \eqref{prop11.5}.
Claim \eqref{prop11.6} follows from the fact that all our homomorphisms are order preserving.
\end{proof}

Some parts of Proposition~\ref{prop11} were observed in \cite{De}.

\subsection{Projection onto the double Catalan monoid}\label{s4.1}

For $\alpha\in \mathcal{DC}_n$ define $\Psi_{\alpha}:=\{w\in S_n: \Psi(z_w)=\alpha\}$.
Recall that a subset of a poset is  {\em convex} if it contains the intervals 
between all comparable points from this subset.
The main combinatorial result on double Catalan monoids is the following:

\begin{proposition}\label{prop7}
Let $\alpha\in \mathcal{DC}_n$.
\begin{enumerate}[$($a$)$]
\item\label{prop7.1} The set $\Psi_{\alpha}$ contains a unique $4321$-avoiding 
permutation $\tau_{\alpha}$.
\item\label{prop7.2} The element $\tau_{\alpha}$ is the unique Bruhat minimal 
element in $\Psi_{\alpha}$.
\item\label{prop7.3} An element $w\in \Psi_{\alpha}$ is Bruhat maximal if and
only if it is $4231$-avoiding. 
\item\label{prop7.4} The set $\Psi_{\alpha}$ is Bruhat convex.
\end{enumerate}
\end{proposition}

\begin{proof}
Using Proposition~\ref{prop6}, the proof of claim \eqref{prop7.1}  is similar 
to the proof of \cite[Lemma~4.21]{Bo}. The rest is similar to the proof 
of Proposition~\ref{prop11}.
\end{proof}

Proposition~\ref{prop7} reduces enumeration of double Catalan monoids to that of 
$4321$-avoiding permutations. There are several formulae (due to I.~Gessel, \cite{Ge}),
enumerating the latter, see \cite[4.4.3]{Bo} for details. Note that 
the set $\Psi_{\alpha}$ might contain several Bruhat maximal elements in general,
see \cite[Theorem~4.18]{Bo}.

\subsection{First derivative of the Kreweras involution}\label{s4.3}

Denote by $\mathfrak{D}_n$ the set of all Dyck paths of semilength
$n$ (i.e., all lattice paths from $(0,0)$ to $(2n,0)$, with steps $(1,1)$ or $(1,-1)$,  
that never go below the $x$-axis). Let 
$\Delta\colon \mathbf{C}_n^+\to \mathfrak{D}_n$ be the usual bijection
defined by outlining, from below, the entries ``$1$'' in the matrix
of an element in $\mathbf{C}_n^+$ and then rotating the path clockwise 
by $135^{\circ}$, as shown in Figure~\ref{fig3}.

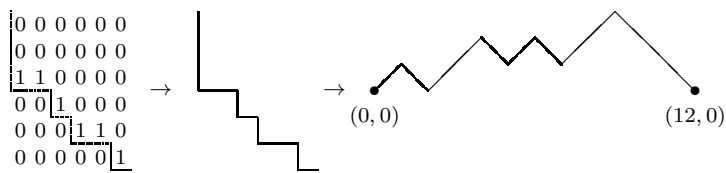
\begin{figure}
\begin{picture}(270.00,80.00)
\put(07.50,15.00){\makebox(0,0)[cc]{\tiny $0$}}
\put(07.50,25.00){\makebox(0,0)[cc]{\tiny $0$}}
\put(07.50,35.00){\makebox(0,0)[cc]{\tiny $0$}}
\put(07.50,45.00){\makebox(0,0)[cc]{\tiny $1$}}
\put(07.50,55.00){\makebox(0,0)[cc]{\tiny $0$}}
\put(07.50,65.00){\makebox(0,0)[cc]{\tiny $0$}}
\put(15.00,15.00){\makebox(0,0)[cc]{\tiny $0$}}
\put(15.00,25.00){\makebox(0,0)[cc]{\tiny $0$}}
\put(15.00,35.00){\makebox(0,0)[cc]{\tiny $0$}}
\put(15.00,45.00){\makebox(0,0)[cc]{\tiny $1$}}
\put(15.00,55.00){\makebox(0,0)[cc]{\tiny $0$}}
\put(15.00,65.00){\makebox(0,0)[cc]{\tiny $0$}}
\put(22.50,15.00){\makebox(0,0)[cc]{\tiny $0$}}
\put(22.50,25.00){\makebox(0,0)[cc]{\tiny $0$}}
\put(22.50,35.00){\makebox(0,0)[cc]{\tiny $1$}}
\put(22.50,45.00){\makebox(0,0)[cc]{\tiny $0$}}
\put(22.50,55.00){\makebox(0,0)[cc]{\tiny $0$}}
\put(22.50,65.00){\makebox(0,0)[cc]{\tiny $0$}}
\put(30.00,15.00){\makebox(0,0)[cc]{\tiny $0$}}
\put(30.00,25.00){\makebox(0,0)[cc]{\tiny $1$}}
\put(30.00,35.00){\makebox(0,0)[cc]{\tiny $0$}}
\put(30.00,45.00){\makebox(0,0)[cc]{\tiny $0$}}
\put(30.00,55.00){\makebox(0,0)[cc]{\tiny $0$}}
\put(30.00,65.00){\makebox(0,0)[cc]{\tiny $0$}}
\put(37.50,15.00){\makebox(0,0)[cc]{\tiny $0$}}
\put(37.50,25.00){\makebox(0,0)[cc]{\tiny $1$}}
\put(37.50,35.00){\makebox(0,0)[cc]{\tiny $0$}}
\put(37.50,45.00){\makebox(0,0)[cc]{\tiny $0$}}
\put(37.50,55.00){\makebox(0,0)[cc]{\tiny $0$}}
\put(37.50,65.00){\makebox(0,0)[cc]{\tiny $0$}}
\put(45.00,15.00){\makebox(0,0)[cc]{\tiny $1$}}
\put(45.00,25.00){\makebox(0,0)[cc]{\tiny $0$}}
\put(45.00,35.00){\makebox(0,0)[cc]{\tiny $0$}}
\put(45.00,45.00){\makebox(0,0)[cc]{\tiny $0$}}
\put(45.00,55.00){\makebox(0,0)[cc]{\tiny $0$}}
\put(45.00,65.00){\makebox(0,0)[cc]{\tiny $0$}}
%%%%%%%%%%%%%%%%%%%%%%%%%%%%%%%%%%%%%%%%%%%%%%%
\dottedline{1}(04.00,70.00)(04.00,40.00)
\dottedline{1}(18.75,40.00)(04.00,40.00)
\dottedline{1}(18.75,40.00)(18.75,30.00)
\dottedline{1}(26.25,30.00)(18.75,30.00)
\dottedline{1}(26.25,30.00)(26.25,20.00)
\dottedline{1}(41.25,20.00)(26.25,20.00)
\dottedline{1}(41.25,20.00)(41.25,10.00)
\dottedline{1}(48.75,10.00)(41.25,10.00)
%%%%%%%%%%%%%%%%%%%%%%%%%%%%%%%%%%%%%%%%%%%%%%%%%
\put(60.00,40.00){\makebox(0,0)[cc]{\tiny $\to$}}
\put(125.00,40.00){\makebox(0,0)[cc]{\tiny $\to$}}
\put(140.00,40.00){\makebox(0,0)[cc]{\tiny $\bullet$}}
\put(260.00,40.00){\makebox(0,0)[cc]{\tiny $\bullet$}}
\put(260.00,30.00){\makebox(0,0)[cc]{\tiny $(12,0)$}}
\put(140.00,30.00){\makebox(0,0)[cc]{\tiny $(0,0)$}}
%%%%%%%%%%%%%%%%%%%%%%%%%%%%%%%%%%%%%%%%%%%%%%%%%
\drawline(74.00,70.00)(74.00,40.00)
\drawline(88.75,40.00)(74.00,40.00)
\drawline(88.75,40.00)(88.75,30.00)
\drawline(96.25,30.00)(88.75,30.00)
\drawline(96.25,30.00)(96.25,20.00)
\drawline(111.25,20.00)(96.25,20.00)
\drawline(111.25,20.00)(111.25,10.00)
\drawline(118.75,10.00)(111.25,10.00)
%%%%%%%%%%%%%%%%%%%%%%%%%%%%%%%%%%%%%%%%%%%%%%%
\drawline(140.00,40.00)(150.00,50.00)
\drawline(160.00,40.00)(150.00,50.00)
\drawline(160.00,40.00)(180.00,60.00)
\drawline(190.00,50.00)(180.00,60.00)
\drawline(190.00,50.00)(200.00,60.00)
\drawline(210.00,50.00)(200.00,60.00)
\drawline(210.00,50.00)(230.00,70.00)
\drawline(260.00,40.00)(230.00,70.00)
%%%%%%%%%%%%%%%%%%%%%%%%%%%%%%%%%%%%%%%%%%%%%%%
\end{picture}
\caption{The map $\Delta$}\label{fig3}
\end{figure}

Define the map $\mathfrak{i}\colon \mathfrak{D}_n\to\mathfrak{D}_n$
as follows: for $\alpha\in\mathbf{C}_n^+$ set 
\begin{equation}\label{eq4}
\mathfrak{i}(\Delta(\alpha)):=
\Delta(\overline{\mathfrak{p}}(z_{\pi_{\alpha}^{-1}})). 
\end{equation}
Note that $\pi_{\alpha}$ was defined to be $321$-avoiding
(see Subsection~\ref{s4.2}). It follows that the inverse 
$\pi_{\alpha}^{-1}$ is $321$-avoiding as well (and hence
coincides with $\pi_{\beta}$ for some $\beta\in \mathbf{C}_n^+$).

An element $\alpha\in\mathbf{C}_n^+$ corresponds, via $\Delta$, to an irreducible 
Dyck path if and only if it has no fixed-point other than $n$. This in turn 
is equivalent to $\pi_{\alpha}$ not having an invariant 
subset of the form $\{1,2,\dots,k\}$ for $k<n$. In particular,
$\pi_{\alpha}$ has no fixed-point. For such $\alpha$
the fact that $\pi_{\alpha}$ is $321$-avoiding can be 
reformulated as follows: given $i,j\in\mathbf{N}$, $i<j$,
then $\pi_{\alpha}(i)<i$ and $\pi_{\alpha}(j)<j$ imply
$\pi_{\alpha}(i)<\pi_{\alpha}(j)$ (and, similarly, 
$\pi_{\alpha}(i)>i$ and $\pi_{\alpha}(j)>j$ imply
$\pi_{\alpha}(i)<\pi_{\alpha}(j)$). This yields that 
on irreducible Dyck paths the map $\mathfrak{i}$ defined in 
\eqref{eq4} coincides with  the first derivative of the involution 
on $\mathfrak{D}_n$ constructed by Kreweras in \cite{Kr}. This extends
to reducible Dyck paths is the obvious way. The derivative appears,
for example, in \cite{BBS}. Thus \eqref{eq4} gives a nice interpretation of
this derivative via inversion of $321$-avoiding permutations.
An example of how this works is given in Figure~\ref{fig4}.

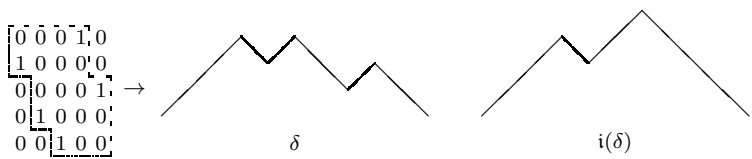
\begin{figure}
\begin{picture}(290.00,70.00)
\put(07.50,15.00){\makebox(0,0)[cc]{\tiny $0$}}
\put(07.50,25.00){\makebox(0,0)[cc]{\tiny $0$}}
\put(07.50,35.00){\makebox(0,0)[cc]{\tiny $0$}}
\put(07.50,45.00){\makebox(0,0)[cc]{\tiny $1$}}
\put(07.50,55.00){\makebox(0,0)[cc]{\tiny $0$}}
\put(15.00,15.00){\makebox(0,0)[cc]{\tiny $0$}}
\put(15.00,25.00){\makebox(0,0)[cc]{\tiny $1$}}
\put(15.00,35.00){\makebox(0,0)[cc]{\tiny $0$}}
\put(15.00,45.00){\makebox(0,0)[cc]{\tiny $0$}}
\put(15.00,55.00){\makebox(0,0)[cc]{\tiny $0$}}
\put(22.50,15.00){\makebox(0,0)[cc]{\tiny $1$}}
\put(22.50,25.00){\makebox(0,0)[cc]{\tiny $0$}}
\put(22.50,35.00){\makebox(0,0)[cc]{\tiny $0$}}
\put(22.50,45.00){\makebox(0,0)[cc]{\tiny $0$}}
\put(22.50,55.00){\makebox(0,0)[cc]{\tiny $0$}}
\put(30.00,15.00){\makebox(0,0)[cc]{\tiny $0$}}
\put(30.00,25.00){\makebox(0,0)[cc]{\tiny $0$}}
\put(30.00,35.00){\makebox(0,0)[cc]{\tiny $0$}}
\put(30.00,45.00){\makebox(0,0)[cc]{\tiny $0$}}
\put(30.00,55.00){\makebox(0,0)[cc]{\tiny $1$}}
\put(37.50,15.00){\makebox(0,0)[cc]{\tiny $0$}}
\put(37.50,25.00){\makebox(0,0)[cc]{\tiny $0$}}
\put(37.50,35.00){\makebox(0,0)[cc]{\tiny $1$}}
\put(37.50,45.00){\makebox(0,0)[cc]{\tiny $0$}}
\put(37.50,55.00){\makebox(0,0)[cc]{\tiny $0$}}
%%%%%%%%%%%%%%%%%%%%%%%%%%%%%%%%%%%%%%%%%%%%%%%
\dottedline{1}(03.00,59.00)(03.00,40.00)
\dottedline{1}(11.25,40.00)(03.00,40.00)
\dottedline{1}(11.25,40.00)(11.25,20.00)
\dottedline{1}(18.75,20.00)(11.25,20.00)
\dottedline{1}(18.75,20.00)(18.75,10.00)
\dottedline{1}(41.25,10.00)(18.75,10.00)
%%%%%%%%%%%%%%%%%%%%%%%%%%%%%%%%%%%%%%%%%%%%%%%
\dashline{2}(03.00,59.00)(33.00,59.00)
\dashline{2}(33.00,40.00)(33.00,59.00)
\dashline{2}(33.00,40.00)(41.25,40.00)
\dashline{2}(41.25,10.00)(41.25,40.00)
%%%%%%%%%%%%%%%%%%%%%%%%%%%%%%%%%%%%%%%%%%%%%%%%%
\put(50.00,35.00){\makebox(0,0)[cc]{\tiny $\to$}}
\put(110.00,15.00){\makebox(0,0)[cc]{\tiny $\delta$}}
\put(230.00,15.00){\makebox(0,0)[cc]{\tiny $\mathfrak{i}(\delta)$}}
%%%%%%%%%%%%%%%%%%%%%%%%%%%%%%%%%%%%%%%%%%%%%%%%%
\drawline(60.00,25.00)(90.00,55.00)
\drawline(100.00,45.00)(90.00,55.00)
\drawline(100.00,45.00)(110.00,55.00)
%%%%%%%%%%%%%%%%%%%%%%%%%%%%%%%%%%%%%%%%%%%%%%%%%

\drawline(130.00,35.00)(110.00,55.00)
\drawline(130.00,35.00)(140.00,45.00)
\drawline(160.00,25.00)(140.00,45.00)
%%%%%%%%%%%%%%%%%%%%%%%%%%%%%%%%%%%%%%%%%%%%%%%
\drawline(180.00,25.00)(210.00,55.00)
\drawline(220.00,45.00)(210.00,55.00)
\drawline(220.00,45.00)(240.00,65.00)
\drawline(280.00,25.00)(240.00,65.00)
%%%%%%%%%%%%%%%%%%%%%%%%%%%%%%%%%%%%%%%%%%%%%%%
\end{picture}
\caption{First derivative of the Kreweras involution}\label{fig4}
\end{figure}

\subsection{Admissible pairs of Dyck paths}\label{s4.4}

An element $\xi\in \mathcal{CB}_n$ is determined by a pair of Dyck paths 
corresponding to $\mathrm{max}(\xi)$ and $\mathrm{max}(\xi^t)$. It is natural
to ask which pairs of Dyck paths correspond to elements of the double Catalan monoid.
Equivalently, given $w\in S_n$ we have the pair of Dyck paths defined as follows: 
\begin{displaymath}
\Delta_w:=(\Delta(\overline{\mathfrak{p}}(z_w)),
\Delta(\overline{\mathfrak{p}}(z_{w^{-1}}))).
\end{displaymath}
A pair of Dyck paths of the form $\Delta_w$ is called {\em admissible}. 
Admissible pairs of Dyck paths were recently described in \cite{BBS}
in terms of the first derivative of the Kreweras involution and a certain partial
order on $\mathbf{C}_n^+$. In the previous subsection we gave an algebraic
interpretation of the first derivative of the Kreweras involution. In this
subsection we give an algebraic interpretation of the partial order
on $\mathbf{C}_n^+$ used in \cite{BBS} and hence provide an algebraic 
interpretation of the main result of \cite{BBS}.

Denote by $\prec$ the order on $\mathbf{C}_n^+$ defined as follows:
for $\alpha,\beta\in \mathbf{C}_n^+$ we set $\alpha\prec\beta$ if and only if 
there exist $\gamma_1,\gamma_2\in \mathbf{C}_n^+$ such that $\beta=\gamma_1\alpha$
and $\beta=\alpha\gamma_2$.  This is the dual of what is classically called the 
$\mathcal H$-order in the semigroup theory literature (here $\mathcal H$ stands for the
corresponding Green's relation).  

Every map $f\colon X\to Y$ defines an equivalence relation $\rho_f$ on $X$, called the
{\em kernel partition of $f$}, as follows; for $a,b\in X$ we have $(a,b)\in \rho_f$
if and only if $f(a)=f(b)$.

Denote by $\prec'$ the preimage under $\Delta$ of the transitive 
closure of the relation $\leq $ on $\mathfrak{D}_n$ defined in \cite{BBS} as follows:
Let $\delta$ be a Dyck path (take, for example,  the left path in Figure~\ref{fig5}).
The bullet points, as in Figure~\ref{fig5}, are called {\em valleys}. To get a cover of
$\delta$ with respect to $\leq $ one is allowed to choose an arbitrary (in particular, empty)
collection of consecutive valleys of $\delta$ and ``rectangularly complete''
them to peaks as shown on the right hand side of Figure~\ref{fig5} (for the second and the third
valleys from the left). Our principal observation here is the following:

\begin{figure}
\begin{picture}(220.00,50.00)
\put(50.00,05.00){\makebox(0,0)[cc]{\tiny $\delta$}}
\put(175.00,05.00){\makebox(0,0)[cc]{\tiny $\delta'$}}
\put(110.00,05.00){\makebox(0,0)[cc]{\tiny $\leq$}}
\put(25.00,20.00){\makebox(0,0)[cc]{\tiny $\bullet$}}
\put(50.00,25.00){\makebox(0,0)[cc]{\tiny $\bullet$}}
\put(65.00,30.00){\makebox(0,0)[cc]{\tiny $\bullet$}}
\put(80.00,25.00){\makebox(0,0)[cc]{\tiny $\bullet$}}
%%%%%%%%%%%%%%%%%%%%%%%%%%%%%%%%%%%%%%%%%%%%%%%
\dottedline{2}(170.00,25.00)(160.00,35.00)
\dottedline{2}(170.00,25.00)(180.00,35.00)
\dottedline{2}(185.00,30.00)(180.00,35.00)
\dottedline{2}(185.00,30.00)(190.00,35.00)
%%%%%%%%%%%%%%%%%%%%%%%%%%%%%%%%%%%%%%%%%%%%%%%%%
\drawline(05.00,20.00)(15.00,30.00)
\drawline(25.00,20.00)(15.00,30.00)
\drawline(25.00,20.00)(40.00,35.00)
\drawline(50.00,25.00)(40.00,35.00)
\drawline(50.00,25.00)(60.00,35.00)
\drawline(65.00,30.00)(60.00,35.00)
\drawline(65.00,30.00)(70.00,35.00)
\drawline(80.00,25.00)(70.00,35.00)
\drawline(80.00,25.00)(85.00,30.00)
\drawline(95.00,20.00)(85.00,30.00)
%%%%%%%%%%%%%%%%%%%%%%%%%%%%%%%%%%%%%%%%%%%%%%%
\drawline(125.00,20.00)(135.00,30.00)
\drawline(145.00,20.00)(135.00,30.00)
\drawline(145.00,20.00)(170.00,45.00)
\drawline(180.00,35.00)(170.00,45.00)
\drawline(180.00,35.00)(185.00,40.00)
\drawline(200.00,25.00)(185.00,40.00)
\drawline(200.00,25.00)(205.00,30.00)
\drawline(215.00,20.00)(205.00,30.00)
%%%%%%%%%%%%%%%%%%%%%%%%%%%%%%%%%%%%%%%%%%%%%%%
\end{picture}
\caption{A cover $\delta'$ of $\delta$ with respect to $\leq$}\label{fig5}
\end{figure}
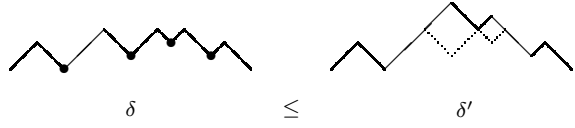

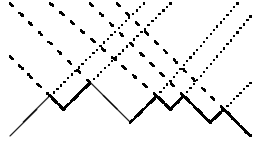
\begin{figure}
\begin{picture}(100.00,60.00)
%%%%%%%%%%%%%%%%%%%%%%%%%%%%%%%%%%%%%%%%%%%%%%%%%
\drawline(05.00,05.00)(20.00,20.00)
\drawline(25.00,15.00)(20.00,20.00)
\drawline(25.00,15.00)(35.00,25.00)
\drawline(50.00,10.00)(35.00,25.00)
\drawline(50.00,10.00)(60.00,20.00)
\drawline(65.00,15.00)(60.00,20.00)
\drawline(65.00,15.00)(70.00,20.00)
\drawline(80.00,10.00)(70.00,20.00)
\drawline(80.00,10.00)(85.00,15.00)
\drawline(95.00,05.00)(85.00,15.00)
%%%%%%%%%%%%%%%%%%%%%%%%%%%%%%%%%%%%%%%%%%%%%%%
\dottedline{2}(20.00,20.00)(55.00,55.00)
\dottedline{2}(35.00,25.00)(65.00,55.00)
\dottedline{2}(60.00,20.00)(90.00,55.00)
\dottedline{2}(70.00,20.00)(95.00,50.00)
\dottedline{2}(85.00,15.00)(95.00,25.00)
%%%%%%%%%%%%%%%%%%%%%%%%%%%%%%%%%%%%%%%%%%%%%%%%%
\dashline{2}(20.00,20.00)(05.00,35.00)
\dashline{2}(35.00,25.00)(05.00,55.00)
\dashline{2}(60.00,20.00)(20.00,55.00)
\dashline{2}(70.00,20.00)(35.00,55.00)
\dashline{2}(85.00,15.00)(45.00,55.00)
%%%%%%%%%%%%%%%%%%%%%%%%%%%%%%%%%%%%%%%%%%%%%%%%%
\end{picture}
\caption{First part of the proof of Proposition~\ref{prop75}}\label{fig6}
\end{figure}

\begin{proposition}\label{prop75}
The relations  $\prec$ and $\prec'$ coincide.
\end{proposition}

\begin{proof}
Let $\alpha\in\mathbf{C}_n^+$. Consider the corresponding Dyck path $\Delta(\alpha)$
(schematically shown as the solid path in Figure~\ref{fig6}). Assume that 
$\alpha\prec\beta$ for some $\beta\in \mathbf{C}_n^+$. Observe that the image of
any element $\alpha\gamma\in \mathbf{C}_n^+$ is a subset of the image of $\alpha$.
This means that every ascent of $\Delta(\beta)$ either overlaps with an
ascent of $\Delta(\alpha)$ or belongs to a dotted line as shown in Figure~\ref{fig6}.
Similarly, the kernel partition defined by any element $\gamma\alpha\in \mathbf{C}_n^+$
is coarser than the kernel partition of $\alpha$. 
This means that every descent of $\Delta(\beta)$ either overlaps with a
descent of $\Delta(\alpha)$ or belongs to a dashed line as shown in Figure~\ref{fig6}.
Hence $\Delta(\beta)$ can be obtained from $\Delta(\alpha)$ by a sequence of operations of
rectangular completion as described above in the definition of $\prec'$. In particular,
$\alpha\prec'\beta$.

Let $\alpha\in\mathbf{C}_n^+$. For $i=1,\dots,n-1$ set $\gamma_i:=\overline{\mathfrak{p}}(e_i)$.
Assume that $\Delta(\alpha)$ is given schematically as shown by the solid path in Figure~\ref{fig7}.
It is easy to check that either $\Delta(\gamma_i\alpha)=\Delta(\alpha)$ or $\Delta(\gamma_i\alpha)$
is obtained by replacing the solid part of the path beneath the dashed path on 
Figure~\ref{fig7} with this dashed path. Similarly, either $\Delta(\alpha\gamma_i)=\Delta(\alpha)$ 
or $\Delta(\alpha\gamma_i)$ is obtained by replacing the solid part of the path beneath the dotted 
path on  Figure~\ref{fig7} with this dotted path.

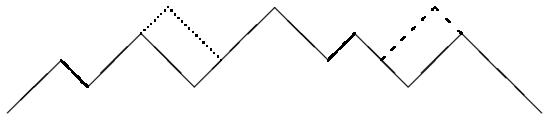
\begin{figure}
\begin{picture}(210.00,60.00)
%%%%%%%%%%%%%%%%%%%%%%%%%%%%%%%%%%%%%%%%%%%%%%%%%
\drawline(10.00,10.00)(30.00,30.00)
\drawline(40.00,20.00)(30.00,30.00)
\drawline(40.00,20.00)(60.00,40.00)
\drawline(80.00,20.00)(60.00,40.00)
\drawline(80.00,20.00)(110.00,50.00)
\drawline(130.00,30.00)(110.00,50.00)
\drawline(130.00,30.00)(140.00,40.00)
\drawline(160.00,20.00)(140.00,40.00)
\drawline(160.00,20.00)(180.00,40.00)
\drawline(210.00,10.00)(180.00,40.00)
%%%%%%%%%%%%%%%%%%%%%%%%%%%%%%%%%%%%%%%%%%%%%%%
\dottedline{2}(60.00,40.00)(70.00,50.00)
\dottedline{2}(90.00,30.00)(70.00,50.00)
%%%%%%%%%%%%%%%%%%%%%%%%%%%%%%%%%%%%%%%%%%%%%%%%%
\dashline{2}(150.00,30.00)(170.00,50.00)
\dashline{2}(180.00,40.00)(170.00,50.00)
%%%%%%%%%%%%%%%%%%%%%%%%%%%%%%%%%%%%%%%%%%%%%%%%%
\end{picture}
\caption{Second part of the proof of Proposition~\ref{prop75}}\label{fig7}
\end{figure}

Now let $\delta$ and $\delta'$ be two Dyck paths such that 
$\delta\leq \delta'$, $\alpha=\Delta^{-1}(\delta)$ and $\beta=\Delta^{-1}(\delta')$.
By an inductive application of the previous paragraph we can find
$\gamma,\gamma'\in \mathbf{C}_n^+$ such that $\beta=\gamma\alpha$ and $\beta=\alpha\gamma'$,
which implies that $\alpha\prec\beta$. As $\prec'$ is a transitive closure of the preimage of
$\leq$ under $\Delta$, it follows that $\alpha\prec'\beta$ implies
$\alpha\prec\beta$, completing the proof.
\end{proof}

As a corollary we can reformulate \cite[Theorem~6]{BBS} as follows:

\begin{corollary}\label{cor76}
A pair $(\delta,\delta')$ of Dyck paths is admissible if and only if 
$\Delta^{-1}(\mathfrak{i}(\delta))\prec \Delta^{-1}(\delta')$ and
$\Delta^{-1}(\mathfrak{i}(\delta'))\prec \Delta^{-1}(\delta)$.
\end{corollary}

\subsection{Self-dual elements}\label{s4.5}

It turns out that self-dual elements of $\mathcal{DC}_{n}$
also admit a very nice combinatorial interpretation. 
Note that the involution on $\mathcal{DC}_{n}$ is the restriction of the matrix transpose, and 
hence the self-dual elements of $\mathcal{DC}_{n}$ are exactly those given by symmetric matrices.
In particular, this includes the $2^{n-1}$ idempotents. Note that idempotents of $\mathcal{DC}_{n}$
are exactly direct sums of matrices consisting entirely of $1$s.

\begin{proposition}\label{selfdual}
Let $w\in S_n$ be a $4321$-avoiding permutation. Then $\Psi(z_w)^t=\Psi(z_w)$ if and only if $w$ is an involution.
\end{proposition}

\begin{proof}
The ``if'' statement follows directly from the fact that $\Psi$ is a homomorphism of involutive
semigroups.
The ``only if'' statement follows from the same fact and the additional observation that 
the inverse of a $4321$-avoiding permutation is $4321$-avoiding.
\end{proof}

It is well-known (see e.g. \cite{BBS0} and references therein), that $4321$-avoiding involutions 
in $S_n$ are in bijection with Motzkin paths of length $n$
(i.e., all lattice paths from $(0,0)$ to $(n,0)$, with steps $(1,1)$, $(1,-1)$ or $(1,0)$,  
that never go below the $x$-axis). In particular, it follows from 
Proposition~\ref{selfdual} that the number of self-dual elements of $\mathcal{DC}_{n}$ equals the 
$n$-th Motzkin number $M_n$ (sequence A001006 in \cite{integer}).

\section{A presentation of the double Catalan monoids}\label{s5}

Our goal in this section is to give a finite presentation of the double Catalan monoid $\mathcal{DC}_{n}$.  To do this, we will take advantage of a result of Tenner in \cite{T}, generalizing a celebrated result in \cite{BJS}.  To state her result, we need to introduce some notation.  Let $s_{i_{1}}s_{i_{2}}\cdots s_{i_{r}}$ be a reduced decomposition of a permutation $w$.  Then if $m\ge 0$, define the $m$-shift of $w$ to be the permutation with reduced decomposition $s_{i_{1}+m}s_{i_{2}+m}\cdots s_{i_{r}+m}$ (this may be defined in a larger symmetric group).  For example, the permutation $4321$ has reduced decomposition 
\begin{equation}\label{reddec4321}
s_{1}s_{2}s_{3}s_{1}s_{2}s_{1}.
\end{equation}
Its $2$-shift is the permutation with reduced decomposition
\begin{displaymath}
s_{3}s_{4}s_{5}s_{3}s_{4}s_{3},
\end{displaymath}
which is $126543$.

A permutation is called {\em vexillary}  if it is $2143$-avoiding.  
For example, the permutations $321$ and  $4321$ are vexillary.
Tenner established the following characterization of vexillary permutations~\cite[Theorem~3.8]{T}.  

\begin{theorem}\label{tennerthm}
A permutation $w$ is vexillary if and only if, for any permutation $v$ containing $w$ as a pattern, some reduced 
decomposition of $v$ contains an $m$-shift (for some $m\ge 0$) of some reduced decomposition of $w$ as a factor.
\end{theorem}

For example, the vexillary permutation $321$ has reduced decomposition $s_{1}s_{2}s_{1}$.  It follows from Theorem~\ref{tennerthm} that a permutation is $321$-avoiding if and only if it has no reduced decomposition containing a factor of the form $s_{i}s_{i+1}s_{i}$, a result first proved in~\cite{BJS}.  A reduced decomposition of the vexillary permutation $4321$ is given in \eqref{reddec4321}.  We deduce that a permutation is $4321$-avoiding if and only it has no reduced decomposition containing a factor of the form 
\begin{equation}\label{4321eq}
s_{i}s_{i+1}s_{i+2}s_{i}s_{i+1}s_{i}.
\end{equation}

We are now in a position to provide our presentation for the double Catalan monoid.

\begin{theorem}\label{present}
The monoid $\mathcal {DC}_{n}$ admits a presentation with generating set $f_{i}$, $i\in \mathbf N'$, and defining relations (for all appropriate $i,j$)
\begin{gather}\label{present.1}
f_{i}^{2}=f_{i};\\ \label{present.2}
f_{i}f_{j}=f_{j}f_{i},\quad i\ne j\pm 1;\\ \label{present.3}
f_{i}f_{i+1}f_{i}=f_{i+1}f_{i}f_{i+1}; \\ \label{present.4}
f_{i}f_{i+1}f_{i+2}f_{i+1}f_{i}=f_{i}f_{i+1}f_{i+2}f_{i}f_{i+1}f_{i}.
\end{gather}
\end{theorem}
\begin{proof}
Let $M$ be the monoid with the above presentation.  
Consider the assignment mapping the generators $f_{i}$ to $\varepsilon_{i}$.  We already know that the 
elements $\varepsilon_{i}$ satisfy \eqref{present.1}--\eqref{present.3}.  It remains to check that they satisfy \eqref{present.4}.  Observe that $s_{i}s_{i+1}s_{i+2}s_{i+1}s_{i}$ is a reduced decomposition of the transposition $u=(i,i+3)$, whereas $s_{i}s_{i+1}s_{i+2}s_{i}s_{i+1}s_{i}$ is a reduced decomposition of the permutation $w=(i,i+3)(i+1,i+2)$.  Therefore, $\alpha_{u}=\alpha_{w}$ as both functions send $i,i+1,i+2,i+3$ to $i+3$ and fix all other elements.  Similarly, $\beta_{u}=\beta_{w}$ as both functions send $i,i+1,i+2,i+3$ to $i$ and fix all other elements. Proposition~\ref{prop6} now yields that the $\varepsilon_{i}$ satisfy \eqref{present.4}.  Thus the quotient map  $\mathcal H_{n}\to \mathcal{DC}_{n}$ factors through $M$.

In light of Proposition~\ref{prop7}\eqref{prop7.1}, to prove the theorem, it suffices to prove the following: if $\lambda\colon \mathcal H_{n}\to M$ denotes the projection, then for each $z_{w}\in \mathcal H_{n}$ there is a $4321$-avoiding permutation $u$ with $\lambda(z_u)=\lambda(z_w)$.  We prove this by induction on length, the case $\mathfrak l(w)=0$ being trivial.  If $w$ is $4321$-avoiding, there is nothing to prove.  Otherwise, Theorem~\ref{tennerthm} implies that $w$ has a reduced decomposition containing as a factor  a shift of some reduced decomposition of $4321$.  Applying braid relations, we may assume it contains a factor of the form \eqref{4321eq}.  Application of a relation of the form \eqref{present.4} yields a shorter length permutation $w'$ such that $\lambda(z_{w'})=\lambda(z_w)$.  The claim follows.
\end{proof}

Note that in the presence of relations \eqref{present.2}--\eqref{present.3} one has that \eqref{present.4} is self-dual since the right hand side can be changed to its reversal using braid relations.

\section{Generalization to other Coxeter groups}\label{s6}

\subsection{$0$-Hecke monoid}\label{s6.1}

Let $(W,S)$ be a Coxeter system; so $W$ is a Coxeter group and $S$ is the set of 
simple reflections or Coxeter generators.  The corresponding $0$-Hecke monoid $\mathcal H(W,S)$ is the 
monoid generated by a set of idempotents $e_s$, indexed by $s\in S$, subject to the 
braid relations of $W$. For example, we have 
$\mathcal{H}_n\cong \mathcal{H}(S_n,\{s_1,\dots,s_{n-1}\})$. 
We often write $\mathcal H(W)$ if $S$ is understood.  One calls $|S|$ the {\em rank} of 
$W$  and denotes it by $\mathbf{r}(W)$.

There are realizations of $\mathcal H(W)$ as the monoid of principal Bruhat ideals 
and the monoid generated by foldings along the walls of the fundamental chamber of 
the Coxeter complex $\Sigma(W)$, exactly as in the case of type $A$ Coxeter groups
(see Section~\ref{s2}).  
Also we have the canonical bijection $w\mapsto z_w$ between $W$ and $\mathcal H(W)$, 
as in the case of type $A$. We denote by $\mathfrak{l}$ the length function for $W$.  
We shall frequently use that 
\begin{equation}\label{foldingrule}
e_sz_w = 
\begin{cases}z_{sw}, & \text{if}\,\, \mathfrak{l}(sw)> \mathfrak{l}(w);
\\ z_w, & \text{if}\,\, \mathfrak{l}(sw)< \mathfrak{l}(w);
\end{cases}
\end{equation}
and dually for right multiplication.

If $J\subseteq S$, it will be convenient to denote by $W_{J}$ the corresponding 
parabolic subgroup of $W$ generated by $J$. In the case that $W_J$ is finite, it 
has a longest element, denoted $w_J$, which moreover is an involution.  The 
corresponding element $z_{w_J}$ is an idempotent that we write $e_J$ and all 
idempotents of $\mathcal H(W)$ are of this form.  In particular, when $W$ is finite, 
then there are $2^{\mathbf{r}(W)}$ idempotents in $\mathcal H(W)$.  It is usual to denote the 
longest element of a finite Coxeter group by $w_0$.  Let us therefore denote the 
corresponding idempotent of $\mathcal H(W)$ by $e_0$.  Notice that $e_0$ is the zero 
of $\mathcal H(W)$.  Observe that $e_J\leq e_K$ if and only if $K\subseteq J$, where 
we recall that idempotents in a semigroup are ordered by $e\leq f$ if and only if 
$ef=e=fe$.

If $w\in W$, then the left and right descent sets of $w$ are the respective sets 
\begin{align*}
D_L(w) &= \{s\in S: \mathfrak{l}(sw)< \mathfrak{l}(w)\} = \{s\in S: e_sz_w=z_w\},\\
D_R(w) &= \{s\in S: \mathfrak{l}(ws)< \mathfrak{l}(w)\} = \{s\in S: z_we_s=z_w\}.\\
\end{align*}
Notice that when $W$ is finite, one has that $D_L(w)=J$ if and only if $e_J$ is the 
unique minimal idempotent stabilizing $z_w$ on the left (and similarly for the right).

\subsection{Analogues of the Catalan and Double Catalan monoids}\label{s6.2}
Let $(W,S)$ be a Coxeter system.   Let us set $(s):=S\setminus \{s\}$ for $s\in S$.  We associate analogues of the Catalan monoid to each finite parabolic subgroup $W_J$ of $W$. Let us fix such a finite parabolic for the course of this subsection.

It is well known that each coset $wW_J$ in $W/W_J$ contains a 
unique element $w^J$ of maximum length.  The set of all such longest coset representatives  is denoted $W^J$.  One has that $w\in W^J$ if and only if $D_R(w)\supseteq J$ (see \cite[Corollary~2.4.5]{BB}).  One 
therefore has the following reformulation of this combinatorics in the language of $0$-Hecke monoids.

\begin{proposition}\label{maxlengthreps}
Let $J\subseteq S$.  Then $\mathcal H(W)e_J=\{z_w: w\in W^J\}$.  
More precisely, one has that $z_we_J = z_{w^J}$.
\end{proposition}

\begin{proof}
Suppose that  $z_we_J=z_u$.  Then clearly $D_R(u)\supseteq J$ and so $u\in W^{J}$.  
On the other hand, from the dual of \eqref{foldingrule} it is immediate that $u\in wW_J$.
\end{proof}

As usual, we view $\mathcal H(W)$ as an ordered monoid where $z_u\leq z_v$ if $u\leq v$ in the Bruhat order.   Since the Bruhat order is compatible with multiplication, one immediately recovers from Proposition~\ref{maxlengthreps} the  well-known fact that the mapping $w\mapsto w^J$ is order preserving (see \cite[Chapter 2, Exercise 16]{BB}). The Bruhat order on the quotient $W/W_J$ is usually defined via the bijection with minimal coset representatives (which are ordered by the Bruhat order), see \cite[2.5]{BB}.  However, in the case that $W_J$ is finite, one can instead use the Bruhat ordering on maximal coset representatives and obtain the same poset structure~\cite[Chapter 2, Exercise 16]{BB}.  Thus as a poset we can identify $W/W_J$ and $\mathcal H(W)e_J$. 

The action of $\mathcal H(W)$ on the left ideal $\mathcal H(W)e_J$ is by order preserving and non-decreasing functions.  This follows immediately from the fact that 
$\mathcal H(W)$ is an ordered monoid in which the identity is minimal.

For example, if $W=S_n$ and $J=(s_{n-1})$, then the corresponding parabolic is $S_{n-1}$ and the maximal coset representatives are the permutations of the form $kn(n-1)\cdots \widehat k\cdots 1$ where $\widehat k$ means omit $k$.  Identifying $k$ with the coset of $kn(n-1)\cdots \widehat k\cdots 1$, we find that the Bruhat ordering is the usual ordering on $\mathbf N$.  The action of $\mathcal H(S_n)$ on $S_{n}/S_{n-1}$ thus factors through the Catalan quotient.    Let us therefore define the {\em generalized Catalan quotient} $\mathcal C(W)_J$ of $\mathcal H(W)$ to be the quotient acting effectively on $\mathcal H(W)e_J\cong W/W_J$ by order preserving and non-decreasing functions.  One has for example that $\mathcal{C}^+_n\cong\mathcal C(S_n)_{(s_{n-1})}$.

To construct analogues of the double Catalan monoids, let us consider the following general situation.  Let $\rho\colon W\to S_n$ be any permutation representation.  Then composing with the standard homomorphism $\Phi\colon S_n\to B_n$ yields a representation $\Phi\rho$ of $W$ by binary relations.  This induces a semiring homomorphism $\mathcal P(W)\to B_n$, which can then be restricted to a homomorphism of ordered monoids $\mathcal H(W)\to B_n$.  The image of a generator $e_s$ is $\Phi\rho(\mathrm{id})+\Phi\rho(s)$. 

Of particular interest is the case when $\rho$ is associated to the action of a finite Coxeter group $W$ on the cosets $W/W_J$.  For example, the double Catalan monoid arises from considering the permutation representation of $S_n$ on $\mathbf N$, which can be identified with the action of $S_n$ on the cosets of its parabolic subgroup $S_{n-1}$.  Hence there is in general a double Catalan quotient $\mathcal D\mathcal C(W)_J$ associated to a finite Coxeter group $W$ and a parabolic subgroup $W_J$ by applying the above construction to the permutation representation associated to the action of $W$ on $W_J$.  With this notation $\mathcal{DC}_{n}=\mathcal{DC}(S_{n})_{(s_{n-1})}$.

There is an alternative, more conceptual, viewpoint on this construction.  We have identified $\mathcal H(W)$ with a submonoid of $\mathcal P(W)$.  But also, $W$ is a subgroup of $\mathcal P(W)$ by identifying elements of $W$ with singleton subsets (in fact $W$ is the group of invertible elements of $\mathcal P(W)$).  Therefore, $\mathcal H(W)$ acts on the left of $\mathcal P(W)$ and $W_J$ acts on the right of $\mathcal P(W)$ by endomorphisms of the additive structure of $\mathcal P(W)$, and these actions commute.  Thus $\mathcal P(W)/W_J= \mathcal P(W/W_J)$ is acted upon by $\mathcal H(W)$ by endomorphisms preserving the additive structure.  This yields a representation of $\mathcal H(W)$ in $B_n$ where $n=[W:W_J]$.  The corresponding effective quotient is $\mathcal D\mathcal C(W)_J$. This construction works
also in the case of infinite $W$.
\vspace{1cm}

\section{Minimal dimension of an effective representation}\label{s7}

\subsection{$0$-Hecke monoid}\label{s7.1}

From now on we assume that $W$, and hence, $\mathcal H(W)$ is finite.
Fix a field $\Bbbk$. Our goal is to compute the minimal degree (dimension) of an effective 
(i.e., injective) linear representation of $\mathcal H(W)$ over $\Bbbk$.  In fact we show 
that there is a unique minimal effective $\mathcal H(W)$-module in the sense that 
it appears as a submodule of every effective module. Note that an effective
representation of a semigroup does not have to give a faithful representation of the
corresponding semigroup algebra.

To provide the intuition for the answer for the monoid $\mathcal{H}(W)$, let us define 
\begin{displaymath}
\mathbf{v}(W)=\sum_{s\in S}[W:W_{(s)}].
\end{displaymath}
Note that $\mathbf{v}(W)$ is the number of vertices of the Coxeter complex 
$\Sigma(W)$. Since $\mathcal H(W)$ acts effectively by type-preserving simplicial maps 
on $\Sigma(W)$, it follows that it acts effectively on the vertex set of 
$\Sigma(W)$.  If $F_{(s)}$ is the vertex of the fundamental chamber with stabilizer $W_{(s)}$, 
then it is easy to see that the opposite vertex $w_0F_{(s)}=e_0F_{(s)}$ is fixed by 
$\mathcal H(W)$.  This fixed element provides a direct summand of the $\Bbbk\mathcal H(W)$-module 
$\Bbbk[W/W_{(s)}]$ isomorphic to the trivial representation.  Killing off this trivial summand 
yields a module of dimension $[W:W_{(s)}]-1$.  The direct sum of these modules over all 
$s\in S$ is an effective $\mathcal H(W)$-module of dimension $\mathbf{v}(W)-\mathbf{r}(V)$.  For example,  we saw that the vertices of $\Sigma(S_{n})$ are the non-empty proper subsets of $\mathbf N$ and so the corresponding representation of $\mathcal H_{n}$ has dimension $2^n-n-1$.

Our main result of this section shows that the module constructed in the previous paragraph is a submodule of 
all other effective modules and hence the minimal dimension of an effective 
$\mathcal H(W)$-module is $\mathbf{v}(W)-\mathbf{r}(V)$. 

The key ingredient of the proof is the following lemma used by Kim and Roush in \cite{KR}, which they attribute to George Bergman.

\begin{lemma}\label{bergmanlemma}
Let $M$ be a monoid and $X\subseteq M$.  Let $L$ be a left ideal of $\Bbbk M$ with simple socle 
and suppose that the socle of $L$ contains a non-zero element of the form $x-y$ with $x,y\in X$.  
Then any $\Bbbk M$-module $V$ that affords a representation whose restriction to $X$ is 
injective contains $L$ as a submodule.
\end{lemma}

\begin{proof}
As $x-y$ does not annihilate $V$, there is an element $v\in V$ such that $(x-y)v\neq 0$.  
The module homomorphism $L\to \Bbbk Mv$ given by $a\mapsto av$ must be injective because 
it does not annihilate the simple socle $\Bbbk M(x-y)$ of $L$.
\end{proof}
 
Next we rephrase the effective action of $\mathcal H(W)$ on the vertices of the Coxeter complex in the language of left ideals. 

\begin{corollary}\label{effectiveonsets}
The action of $\mathcal H(W)$ on the left ideal 
\begin{displaymath}
\bigcup_{s\in S} \mathcal H(W)e_{(s)}
\end{displaymath} 
is effective.
\end{corollary}

\begin{proof}
First note that, for $v,w\in W$, one has $v=w$ if and only if $vW_{(s)}=wW_{(s)}$ for all 
$s\in S$ because only the identity belongs to every maximal parabolic subgroup.  
Proposition~\ref{maxlengthreps} thus yields $z_{v}=z_{w}$ if and only if 
$z_{v}e_{(s)}=z_{w}e_{(s)}$ for all $s\in S$.  This establishes the corollary.
\end{proof}

The following proposition is elementary.

\begin{proposition}\label{maxcosetmoves}
Let $J\subseteq S$.
If $w\in W^{J}$ and $s\in S$ with $swW_{J}\ne wW_{J}$, then $sw\in W^{J}$.
\end{proposition}

\begin{proof}
Clearly, if $\mathfrak{l}(sw)>\mathfrak{l}(w)$ with $w\in  W^{J}$, then 
$z_{sw}=e_{s}z_{w}\in \mathcal H(W)e_{J}$ and so Proposition~\ref{maxlengthreps} implies 
$sw\in W^{J}$.  Conversely, suppose that $\mathfrak{l}(sw)<\mathfrak{l}(w)$.  Then $sw< w$ 
in the Bruhat order.  Thus $(sw)^{J}\le w^J=w$ by the remark after Proposition~\ref{maxlengthreps}.  
The inequality is in fact strict by the hypothesis.  Thus we have 
$\mathfrak{l}(w)-1=\mathfrak{l}(sw)\leq \mathfrak{l}((sw)^{J})< \mathfrak{l}(w)$ and so $sw=(sw)^{J}$.
\end{proof}

Recall that in a finite Coxeter group one has $w_0Sw_0=S$ (it is easy to see that 
$\mathfrak{l}(w_0sw_0)=1$ for any $s\in S$).

\begin{proposition}\label{maxcoset}
Let $s\in S$.  Then $w_0s$ is the unique element of $W^{(s)}$ covered by $w_0$ in the 
Bruhat order.  Moreover, $D_L(w_0s) = (w_0sw_0)$. 
\end{proposition}

\begin{proof}
The elements covered by $w_0$ in the Bruhat order are the $w_0t$ with $t\in S$.  
But if $t\neq s$, then $w_0W_{(s)}=w_0tW_{(s)}$ yielding the first statement.  
For the second, if $\mathfrak{l}(tw_0s)>  \mathfrak{l}(w_0s)$,
$t\in S$, then $tw_0s=w_0$ and so $t=w_0sw_0$.
\end{proof}

Let us put $P_{(s)} :=\Bbbk\mathcal H(W)e_{(s)}$.  This is a projective $\Bbbk \mathcal H(W)$-module.  
It contains the trivial submodule $\Bbbk e_0$ as a direct summand and we have
\begin{equation}\label{directsumdecomp}
P_{(s)}= \Bbbk\mathcal H(W)(e_{(s)}-e_0)\oplus \Bbbk e_0.
\end{equation}
As the idempotent $e_{(s)}$ is $0$-minimal, the idempotent $e_{(s)}-e_0$ is primitive.
Set $P'_{(s)}:=\Bbbk \mathcal H(W)(e_{(s)}-e_0)$; it is a projective indecomposable module of 
dimension $[W:W_{(s)}] -1$.

The irreducible representations of $\mathcal H(W)$ are well known.  For each 
$J\subseteq S$ there is an irreducible representation $\theta_J\colon \mathcal H(W)\to 
\mathrm{End}_{\Bbbk}(\Bbbk)$ given by 
\begin{displaymath}
\theta_J(z_w) = 
\begin{cases}
\mathrm{id}, & \text{if}\,\, w\in W_J;\\ 0, & \text{else};
\end{cases} 
\end{displaymath}
and these are all the irreducible representations (see \cite{No,Ca} for details).  
In particular, the irreducible representations do not help to find an effective representation.

The main technical result of this section is the following theorem.

\begin{theorem}\label{simplesocle}
The module $P'_{(s)}$ has simple socle $\Bbbk (z_{w_{0}s}-e_{0})$
isomorphic to $\theta_{(w_0sw_0)}$.
\end{theorem}

The proof of this theorem proceeds via several lemmas.  

\begin{lemma}\label{insocle}
The span of the vector $z_{w_0s}-e_0$ is isomorphic to $\theta_{(w_0sw_0)}$.
\end{lemma}

\begin{proof}
Proposition~\ref{maxcoset} implies that for $t\in S$, one has 
\begin{displaymath}
e_t(z_{w_0s}-e_0) =  \begin{cases}
z_{w_0s}-e_0, & \text{if}\ t\neq w_0sw_0; \\ 0,      & \text{if}\ t=w_0sw_0;
\end{cases}
\end{displaymath}
as required.
\end{proof}

By the direct sum decomposition \eqref{directsumdecomp}, it suffices to prove that 
the socle of $P_{(s)}$ is $\Bbbk(z_{w_0s}-e_0)\oplus \Bbbk e_0$.  To do this we need to 
perform a detailed analysis of the eigenspaces of the elements $e_t$, $t\in S$.  For an 
element 
\begin{displaymath}
v=\sum_{w\in W^{(s)}}c_wz_w\in P_{(s)},
\end{displaymath}
let $\mathrm{supp}(v)$ be 
the {\em support} of $v$ (i.e., the set of $w$ with $c_w\neq 0$).

\begin{lemma}\label{eigenspaces}
Let $t\in T$ and $v\in P_{(s)}$.  Then:
\begin{enumerate}[$($i$)$]
\item\label{eigenspaces.1} $e_tv=v$ if and only if $t\in D_L(w)$ for all $w\in \mathrm{supp}(v)$;
\item\label{eigenspaces.2} $e_tv=0$ if and only if the following two conditions are satisfied:
\begin{enumerate}[$($a$)$]
\item\label{eigenspaces.2.1} $\{wW_{(s)}: w\in \mathrm{supp}(v)\}$ is a union of 
two-element orbits of $t$;
\item\label{eigenspaces.2.2} if $\{wW_{(s)},twW_{(s)}\}$ is a two-element 
orbit, then $c_{w^{(s)}}=-c_{(tw)^{(s)}}$.
\end{enumerate}
\end{enumerate}
\end{lemma}

\begin{proof}
Claim \eqref{eigenspaces.1} is clear since the eigenspace of $1$ for $e_t$ is the subspace 
$\Bbbk e_t\mathcal H(W)e _{(s)}$, which has basis consisting of the $z_w$ such that 
$w\in W^{(s)}$ and $t\in D_L(w)$.

For claim \eqref{eigenspaces.2}, we know that the eigenspace of $0$ has basis  consisting of all 
differences $z_w-e_tz _w$ such that $z_w\neq e_tz _w$.  Equivalently, it has basis  consisting of all elements 
of the form $z_w-z_{tw}$ such that $\mathfrak{l}(tw)>\mathfrak{l}(w)$ with $w,tw\in W^{(s)}$. 
Claim \eqref{eigenspaces.2} is now immediate from Proposition~\ref{maxcosetmoves}.
\end{proof}

We can now prove Theorem~\ref{simplesocle}.

\begin{proof}[Proof of Theorem~\ref{simplesocle}]
Suppose that $v\in P_{(s)}$ generates the irreducible representation $\theta_J$. Then $v$ is 
fixed by each $e_t$, $t\in J$, and annihilated by each $e_t$, $t\notin J$.  Thus the set $X$ of cosets of 
the form $wW_{(s)}$ with $w\in \mathrm{supp}(v)$ is $W_{S\setminus J}$-invariant by 
Lemma~\ref{eigenspaces}\eqref{eigenspaces.2.1}.  
Suppose that $\mathcal O$ is a $W_{S\setminus J}$-orbit on $X$ and that $w\in W^{(s)}$ is maximal 
with respect to the Bruhat order such that $wW_{(s)}\in \mathcal O$.  Then by 
Lemma~\ref{eigenspaces}\eqref{eigenspaces.1} one has $J\subseteq D_{L}(w)$.  On the other hand, since 
$twW_{(s)}\in \mathcal O$ for all $t\in S\setminus J$, by the maximality of $w$ 
we have $S\setminus J\subseteq D_{L}(w)$ and so $w=w_{0}$.  
In particular, it follows that $\mathcal O$ is always the orbit of $w_0$ and hence is unique.  
If $J=S$, it then follows that
$v\in \Bbbk e_{0}$ and we are done.  Otherwise, we must have $J=(w_{0}sw_{0})$ because if 
$t\in S\setminus J$ then, according to Lemma~\ref{eigenspaces}\eqref{eigenspaces.2.1}, $w_{0}W_{J}$ 
is not fixed by 
$t$ and so $t=w_{0}sw_{0}$ as a consequence of Propositions~\ref{maxcosetmoves} and 
\ref{maxcoset}.  An application of Lemma~\ref{eigenspaces}\eqref{eigenspaces.2.2} 
yields $v\in \Bbbk(z_{w_{0}s}-e_{0})$, as required.
\end{proof}

We are now in a position to prove the main result of this subsection. Recall that $\mathbf{r}(V)=|S|$ and $\mathbf{v}(W)$ is the number of vertices of the Coxeter complex of $W$, i.e., the sum of the indices of the maximal parabolics.

\begin{theorem}\label{thmnew1}
Let $W$ be a finite Coxeter group with set $S$ of Coxeter generators and let $M$ be an effective 
$\mathcal H(W)$-module over a field $\Bbbk$.  Then $M$ contains a submodule isomorphic 
to the projective module  
\begin{displaymath}
P=\bigoplus_{s\in S} \Bbbk\mathcal H(W)(e_{(s)}-e_{0}).
\end{displaymath} 
Consequently, the minimal dimension of an effective linear representation of 
$\mathcal H(W)$ is $\mathbf{v}(W)-\mathbf{r}(V)$.
\end{theorem}

\begin{proof}
By Lemma~\ref{bergmanlemma} and Theorem~\ref{simplesocle} each module $P'_{(s)}$ is isomorphic 
to a submodule of $M$.  Moreover, since the simple socles of the $P'_{(s)}$, $s\in S$, are 
pairwise non-isomorphic, it follows that the intersection of $P'_{(s)}$ and $P'_{(t)}$ is trivial, for $s\neq t$, and so $P$ is a submodule of $M$.

It thus remains to show that $P$ is effective.  But this follows from Corollary~\ref{effectiveonsets}.
\end{proof}

As a consequence of Corollary~\ref{effectiveonsets} and Theorem~\ref{thmnew1}, we may deduce that the minimum degree of an effective action of $\mathcal H(W)$ on a set is $\mathbf{v}(W)-\mathbf{r}(V)+1$.

Specializing to the case $W=S_n$ we obtain the following result.

\begin{corollary}\label{cornew2}
The minimal degree of an effective linear representation of $\mathcal H_{n}$ is $2^{n}-n-1$.
\end{corollary}

\begin{proof}
For $i\in\mathbf{N}$ and $s=s_i$ we have $[S_n:(S_n)_{(s_i)}]=\binom{n}{i}$. The rank of $S_{n}$ is $n-1$.
\end{proof}

\subsection{Double Catalan monoid}\label{s7.2}

Theorem~\ref{simplesocle} also permits us to compute the minimal degree of an effective representation of the double Catalan monoid $\mathcal {DC}_{n}$ over a field $\Bbbk$.  

Proposition~\ref{prop6} implies that the image of an element $z_{w}$ of the $0$-Hecke monoid $\mathcal H_{n}$ in $\mathcal {DC}_{n}$ is determined by its images in $\mathbf C^{+}_{n}$ and $\mathbf C^{-}_{n}$.  These are in turn the effective quotients of $\mathcal H_{n}$ coming from its respective actions on the left ideals $\mathcal H_{n}e_{(s_{1})}$ and $\mathcal H_{n}e_{(s_{n-1})}$.  Thus one can identify $\mathcal {DC}_{n}$ with the effective quotient of the action of $\mathcal H_{n}$ on the left ideal $\mathcal H_{n}e_{(s_{1})}\cup \mathcal H_{n}e_{(s_{n-1})}$.   Notice that this left ideal acts effectively on itself and so projects injectively into $\mathcal {DC}_{n}$.  The action of $\mathcal {DC}_{n}$ on this left ideal can be understood easily in terms of its representation by boolean matrices.  Namely, let $v_{j}$ be the characteristic vector of the subset $\{1,\ldots, j\}$ and let $v'_{j}$ be the characteristic vector of $\{j,\ldots, n\}$.  Then the elements $\{v_1,v_2,\dots,v_n\}$ form an invariant subset on which $\mathcal {DC}_{n}$ acts as $\mathbf C^{+}_{n}$ and the elements $\{v'_1,v'_2,\dots,v'_n\}$ form an invariant subset on which $\mathcal {DC}_{n}$ acts as $\mathbf C_{n}^{-}$.  The two subsets intersect in the vector $v_{n}=v'_{1}$.

In summary, retaining the notation of the previous subsection, we can view $P=P'_{(s_{1})}\oplus P'_{(s_{n-1})}$ as an effective $\mathcal {DC}_{n}$-module over $\Bbbk$ of dimension $2n-2$.  We claim $P$ is the unique minimal effective $\mathcal {DC}_{n}$-module. 

\begin{theorem}\label{degreedoublecat}
Let $\Bbbk$ be a field.  Then the minimal dimension of an effective linear representation of $\mathcal {DC}_{n}$ over $\Bbbk$ is $2n-2$.  The unique minimal effective $\mathcal {DC}_{n}$-module is $P$ (defined above).
\end{theorem}
\begin{proof}
Any effective $\mathcal {DC}_{n}$-module $M$ yields a representation of $\mathcal H_{n} $ that is injective on the left ideal 
$\mathcal H_{n}e_{(s_{1})}\cup \mathcal H_{n}e_{(s_{n-1})}$.  Lemma~\ref{bergmanlemma} and Theorem~\ref{simplesocle} now imply that $P'_{(s_{1})}$ and $P'_{(s_{n-1})}$ are submodules of $M$.  The argument in the proof of Theorem~\ref{thmnew1} then shows that their direct sum $P$ is a submodule of $M$.  Thus $P$ is the unique minimal effective $\mathcal {DC}_{n}$-module.  As it has dimension $2n-2$, this completes the proof.
\end{proof}

\vspace{0.2cm}

\noindent
V.M.: Department of Mathematics, Uppsala University, Box 480,
SE-75106, Uppsala, SWEDEN, e-mail: {\tt mazor\symbol{64}math.uu.se},\\
web: ``http://www.math.uu.se/$\sim$mazor/''
\vspace{0.5cm}

\noindent
B.S.: School of Mathematics and Statistics; Carleton University,
1125 Colonel By Drive, Ottawa, Ontario K1S 5B6, CANADA,\\
e-mail: {\tt bsteinbg\symbol{64}math.carleton.ca}\\
web: ``http://www.math.carleton.ca/$\sim$bsteinbg/''
\vspace{0.1cm}

\noindent
Current address: Department of Mathematics, City College of New York, NAC 8/133,
Convent Ave at 138th Street, New York, NY 10031, USA\\
e-mail: {\tt  bsteinberg\symbol{64}ccny.cuny.edu}\\
web: ``http://www.sci.ccny.cuny.edu/$\sim$benjamin/''
\end{document}